\documentclass[11pt,a4paper,reqno]{amsart}

\usepackage{amsfonts}
\usepackage{amsmath}
\usepackage{amssymb}
\usepackage{graphicx}
\usepackage{graphics}
\usepackage{epsfig}
\usepackage{anysize}
\usepackage{tabularx} 
\usepackage{ragged2e,listings} 
\usepackage{booktabs,cleveref}
\usepackage{float}
\usepackage{lipsum}
\usepackage{wrapfig}
\usepackage{tikz}
\marginsize{3cm}{3cm}{3cm}{3cm}
\newtheorem{theorem}{Theorem}[section]
\newtheorem{proposition}{Proposition}[section]

\newtheorem{definition}{Definition}[section]

\makeatletter

\begin{document}

\begin{flushleft}
\end{flushleft}
\setcounter{page}{1}
\vspace{2cm}
\author[]{Subarsha Banerjee }
\title[\centerline{Subarsha Banerjee: The Szeged Index of Power Graph of  Finite Groups
\hspace{0.5cm}}]{The Szeged Index of Power Graph of  Finite Groups}

\thanks{\noindent $^1$ Department of Mathematics, JIS University, 81, Nilgunj Rd, Agarpara, Kol-700109, India \\
\indent \,\,\, e-mail: subarshabnrj@gmail.com/subarsha.banerjee@jisuniversity.ac.in; 0000-0002-7029-7650\\
}

\begin{abstract}
The Szeged index of a graph is an invariant with several applications in chemistry. 
The power graph of a finite group $G$ is a graph having vertex set as $G$ in which two vertices $u$ and $v$ are adjacent if  $v=u^m$ or $u=v^n$ for some $m,n\in \mathbb{N}$.
In this paper, we first obtain a formula for the Szeged index of the generalized join of graphs.
As an application, we obtain the Szeged index of the power graph of the finite cyclic group $\mathbb Z_n$ for any $n>2$.
We further obtain a relation between the Szeged index of the power graph of $\mathbb Z_n$ and the  Szeged index of the power graph of the dihedral group $\mathrm{D}_n$.
We also provide SAGE codes for evaluating the Szeged index of the power graph of $\mathbb{Z}_n$ and $\mathrm{D}_n$ at the end of this paper.
		
\bigskip \noindent Keywords: Szeged index, generalized join, power graph, finite cyclic group, dihedral group.

\bigskip \noindent AMS Subject Classification: 05C09, 05C25 
 
\end{abstract}
\maketitle 
\bigskip
\bigskip
%
\section{Introduction}
	Throughout the paper, $G$ will denote either a finite graph or a finite group, whichever is applicable, and it will be clear from the context.
	Also, given a set $A$, the number of elements in $A$ is denoted by $\lvert A\rvert$.
	
	In theoretical chemistry, topological indices are widely used to understand the physical and chemical properties of chemical compounds. There exist many different types of topological indices that aim to capture various aspects of the molecular graphs associated with the molecules considered.
	The oldest and most thoroughly examined topological index is the Wiener index (see \cite{gutman1993wiener,dobrynin2001wiener,knor2016wiener} for some references).
	Another topological index, the Szeged index, was introduced much later and extensively studied.
	We discuss the Szeged index and its close association with the  Wiener index in short below.
	Suppose  $G$ is a finite, simple, and connected graph with vertex set $V(G)$ and edge set $E(G)$. 	The number of vertices of $G$ is denoted by $\lvert G\rvert $.
	If two vertices $a$ and $b$ are adjacent to each other, we denote it by $a\sim b$.
	Moreover, $E(G)=\{(a,b): a,b\in V(G) \text{ with } a\sim b\}$.
	The distance between two vertices $a,b\in G$, denoted by $d(a,b)$, is defined to be the length of the shortest path from $a$ to $b$. The \textit{Wiener index} \cite{klavvzar1996szeged} of a connected graph $G$  is defined as follows:
	\begin{flalign}
	\label{W}
	W(G)=\frac{1}{2}\sum_{a,b\in V(G)} d(a,b).
	\end{flalign}
	Suppose $e \in E(G)$ is an edge between the vertices $a$ and $b$ of $G$. 
	We first state the following definitions:
	We define \begin{align}
	\label{S}
	\begin{split}
	n_1(ab|G)=\lvert\{x\in G: d(x,a)<d(x,b)\}\rvert\\
	n_2(ab|G)=\lvert\{x\in G: d(x,b)<d(x,a)\}\rvert.
	\end{split}
	\end{align}
	The two quantities described in \Cref{S} were mentioned for the first time in  \cite{wiener1947structural}.
    For a long time, it was known  that the formula
	\begin{flalign}
	\label{AL}
	W(G)=\sum_{e(=ab)\in E(G)} n_1(ab|G)n_2(ab|G)
	\end{flalign}
	holds for molecular graphs of alkanes.
	In \cite{gutman2012mathematical}, it was proved that \Cref{AL} holds for all trees.
	Furthermore, in \cite{dobrynin1994graph} it was shown that \Cref{AL} does not hold in general (in particular for graphs containing cycles), and only holds for graphs whose each block is a complete graph. Although the attempts to change the right-hand side of \Cref{AL} to make it applicable to all connected graphs have been successfully made in \cite{lukovits1992correlation,lukovits1990wiener}, the resulting expressions were very confusing.
    In \cite{gutman1994formula}, it was suggested that the complications arising with the generalization of \Cref{AL} to all connected graphs containing cycles could be overcome by using the right–hand side of \Cref{AL} as the definition of a new graph invariant. Consequently, the formula was extended to all graphs and it came to be known as the Szeged index of a graph. The \textit{Szeged index} of a connected graph $G$ is defined as follows:
	\begin{flalign}
	\label{Sz}
	Sz(G)=\sum_{e(=ab)\in E(G)} n_1(ab|G)n_2(ab|G).
	\end{flalign}
	
	The Szeged index has been considered from multiple viewpoints, see for example \cite{gutman1998szeged,gutman1995szeged} and the references therein for some literature on the same. The Szeged index is closely related to the Wiener index, see for example the references \cite{bonamy2017szeged,chen2012szeged}.
	In \cite{klavvzar1996szeged}, the authors deduced the Szeged index of the Cartesian product of graphs.
	In \cite{khalifeh2008matrix}, the authors determined the Szeged index of join and composition of graphs. The Szeged index of bridge graphs was determined in \cite{mansour2009vertex}. 
    Some variants of the Szeged index namely the edge PI index, edge Szeged index,
edge-vertex Szeged index, vertex-edge Szeged index, and revised edge Szeged index 
under the rooted product of graphs have been studied in \cite{azari2016szeged}.
 Recently, the Szeged index of unicyclic graphs was studied in \cite{liu2022revised}. 
 Motivated by the above works, in this paper, we deduce the Szeged index of the generalized join of graphs.

	The \textit{power graph} \cite{chakrabarty2009undirected} of a finite group $G$  has vertex set as $G$, and two distinct vertices $u,v\in G$  are adjacent if  $u=v^m$ or $v=u^n$ for some $m,n\in \mathbb{N}$.
    Topological indices of graphs associated with algebraic structures have been studied in \cite{rehman2022mostar,abbas2021hosoya,alimon2020szeged,salman2022noncommuting,kharkongor2023topological} recently. 
    Here, in this paper, as an application of our results, we have determined the Szeged index of the power graph of certain finite groups, viz. the finite cyclic group $\mathbb{Z}_n$ and the dihedral group $\mathrm{D}_n$.
	
	The paper has been organized as follows:
	In \Cref{S1}, we describe the Szeged index of the generalized join of graphs in terms of the Szeged index of the constituent graphs. 
	In \Cref{S2}, we find the  Szeged index of the power graph of the finite cyclic group $\mathbb{Z}_n$.
	In \Cref{S3}, we find the  Szeged index of the power graph of the dihedral group $\mathrm{D}_n$ and provide a relationship between the Szeged index of the power graph of  $\mathbb{Z}_n$ and $\mathrm D_n$.
	In \Cref{S4}, we provide the SAGE codes for evaluating the Szeged index of the power graph of the finite cyclic group $\mathbb{Z}_n$, and the dihedral group $\mathrm{D}_n$.

	\section{Szeged Index of Generalized Join of Graphs }
	\label{S1}
	In this section, we deduce the Szeged index of the generalized join of graphs in terms of the constituent graphs.
	Consider a family of $n$ connected graphs $G_i$ such that  $\mid G_i \mid=n_i$ for $1\le i\le n$. Assume that  $\mathcal{G}$ is a graph such that  $\mid \mathcal{G}\mid =n$. 
	The \textit{generalized join} or the $\mathcal{G}$-join of graphs $G_i$ is defined as the following:
	
	\begin{definition}\cite [p. $15$]{schwenk1974characteristic}
		\label{generalizedcomposition}
		Suppose $\mathcal{G}$ is a graph with vertex set $V(\mathcal{G})=\{1,2,\ldots, n\}$.
		Suppose $G_i$ be disjoint connected graphs of order $n_i$  with vertex sets $V(G_i)$ for $1\le i\le n$.
		The $\mathcal{G}$-join of graphs $G_1,G_2,\dots, G_n$ denoted by $G=\mathcal{G}[G_1,G_2,\dots,G_n]$ is formed by taking the graphs $G_i$  and any two vertices $v_i\in G_i$ and $v_j\in G_j$ are adjacent if $i$ is adjacent to $j$ in $\mathcal{G}$. 
		
	\end{definition}
	We now state the main result of this section.
	
	\begin{theorem}
		\label{T1}
		Consider $\mathcal{G}$ to be a graph with vertex set $V(\mathcal{G})=\{1,2,\ldots, n\}$.
		Assume that  $G_i$'s are disjoint connected graphs of order $n_i$  with vertex sets $V(G_i)$ for $1\le i\le n$.
		The Szeged index of  $G=\mathcal{G}[G_1,G_2,\dots,G_n]$ is given as follows:
		\begin{flalign*}
		Sz(G)=\sum_{i=1}^n Sz(G_i)+ 
		\sum_{\substack{i,j=1\\i<j,i\sim j}}^n\biggl[ \biggl(\sum_{\substack{k=1\\k\sim i, k\nsim j}}^n |G_k|+1\biggr)\biggl(\sum_{\substack{k=1\\k\nsim i, k\sim j}}^n |G_k|+1\biggr)\biggr]\mid G_i\mid \mid G_j\mid.
		\end{flalign*}
		
	\end{theorem}
	
	\begin{proof}
		
		Suppose $ab\in E(G)$.
		Since $ab$ is an edge in $G$, two possible cases may arise:
		\begin{enumerate}
			\item [1.] $a\in G_i$ and $b\in G_i$ for some $i$ where  $1\le i\le n$, or
			\item [2.] $a\in G_i$ and $b\in G_j$ for some $1\le i,j\le n$ such that $i\sim j$ in $\mathcal{G}$. 
		\end{enumerate}
		Hence,
		\begin{flalign}
		\label{A}
		\tag{\textcolor{red}{\textbf{A}}}
		\begin{split}
		Sz(G)&=\sum_{e(=ab)\in E(G)} n(ab|G)
		\\
		&=\sum_{i=1}^n \bigg(\sum_{\substack{e(=ab)\in E(G)\\{a,b\in G_i}}} n(ab|G)\bigg) +\sum_{\substack{e(=ab)\in E(G)\\ a\in G_i,b\in G_j\\i\sim j}} n(ab|G)
		\\
		&=\sum_{i=1}^n\bigg(\sum_{\substack{e(=ab)\in E(G)\\{a,b\in G_i}}} n_1(ab|G)n_2(ab|G)\bigg)+\sum_{\substack{e(=ab)\in E(G)\\ a\in G_i,b\in G_j\\i\sim j}} n(ab|G)
		\end{split}
		\end{flalign}
		We shall now evaluate each term of the summation given in \Cref{A} in what follows. 
		
		We have,
		\begin{flalign}
		\label{B}
		\tag{*}
		\begin{split}
		n_1(ab|G)&=\lvert \{x\in G:d(x,a)<d(x,b)\}\rvert
		\\
		&=\lvert\biggl(\{x\in G_i: d(x,a)<d(x,b)\}\cup \{x\in G\setminus G_i:d(x,a)<d(x,b)\}\biggr)\rvert
		\\
		&=\lvert\{x\in G_i: d(x,a)<d(x,b)\}\rvert+\lvert  \{x\in G\setminus G_i:d(x,a)<d(x,b)\}\rvert
		\\
		&=n_1(ab|G_i)+\lvert\{x\in G\setminus G_i:d(x,a)<d(x,b)\}\rvert.
		\end{split}
		\end{flalign}
		We note that  since $a,b\in G_i$, then $d(x,a)=d(x,b)$ for all $x\in G\setminus G_i$.
		Hence we find that \begin{flalign*}
		\{x\in G\setminus G_i:d(x,a)<d(x,b)\}&=\emptyset
		\end{flalign*}
		Thus, we have,
		\begin{flalign}
		\label{C}
		\tag{**}
		\lvert \{x\in G\setminus G_i:d(x,a)<d(x,b)\}\rvert &=0.
		\end{flalign}
		
		Hence, using \Cref{B,C}, we find that for all $1\le i\le n$,
		\begin{equation}
		\label{1}
		n_1(ab|G)=n_1(ab|G_i).
		\end{equation}
		
		Using similar arguments as above, we find that for all $1\le i\le n$,
		\begin{equation}
		\label{2}
		n_2(ab|G)=n_2(ab|G_i).
		\end{equation}
		
		Hence we have,
		\begin{equation}
		\label{3}
		\begin{split}
		\sum_{i=1}^n \bigg(\sum_{\substack{e(=ab)\in E(G)\\{a,b\in G_i}}} n_1(ab|G)n_2(ab|G)\bigg)
		&=\sum_{i=1}^n n_1(ab|G_i)n_2(ab|G_i)
		\\
		&=\sum_{i=1}^n Sz(G_i).
		\end{split}
		\end{equation}

		We now evaluate the second term of the summation given in \Cref{A}.
		
		We have,
		\begin{flalign*}
		&\sum_{\substack{e(=ab)\in E(G)\\ a\in G_i,b\in G_j\\i\sim j}} n(ab|G)
		=\sum_{\substack{e(=ab)\in E(G)\\ a\in G_i,b\in G_j\\i\sim j}} n_1(ab|G)n_2(ab|G) 
		\end{flalign*}
		Here $a\in G_i,b\in G_j$ for some $i,j$ with $i\sim j$.
		Consider the set $\{x\in G:d(x,a)<d(x,b)\}$.
		Suppose $x\in V(G)$ satisfies the relation $d(x,a)<d(x,b)$. Then one possibility is $x=a$.
		Moreover, $x\notin G_i$, as otherwise, since $i\sim j$, $x$ will be adjacent to $b$ which would imply $d(x,b)=1$, which is false.
		Thus, $x$ must belong to all those $G_k$ such that $i\sim k$, but $j\nsim k$ in $\mathcal{G}$.
		
		Thus, 
		\begin{equation}\label{4}
		\begin{split}
		n_1(ab|G)&=\lvert\{x\in G:d(x,a)<d(x,b)\}\rvert
		\\
		&=\sum_{\substack{k=1\\k\sim i, k\nsim j}}^n \mid G_k\mid +1.
		\end{split}
		\end{equation}
		
		Using similar arguments we obtain,
		
		\begin{align}\label{5}
		\begin{split}
		n_2(ab|G)&=\lvert\{x\in G:d(x,b)<d(x,a)\}\rvert
		\\
		&=\sum_{\substack{k=1\\k\nsim i, k\sim j}}^n \mid G_k\mid +1.
		\end{split}
		\end{align}
		
		Using \Cref{4,5} we obtain,
		
		\begin{align}\label{6}
		\begin{split}
		\sum_{\substack{e(=ab)\in E(G)\\ a\in G_i,b\in G_j\\i\sim j}} n(ab|G)
		&=\sum_{\substack{i,j=1\\i<j,i\sim j}}^n\biggl[ \biggl(\sum_{\substack{k=1\\k\sim i, k\nsim j}}^n |G_k|+1\biggr)\biggl(\sum_{\substack{k=1\\k\nsim i, k\sim j}}^n |G_k|+1\biggr)\biggr]\mid G_i\mid \mid G_j\mid.
		\end{split}
		\end{align}
		
		Putting the values obtained from \Cref{3,6} in \Cref{A}, we obtain
		\begin{flalign*}
		Sz(G)=\sum_{i=1}^n Sz(G_i)+ 
		\sum_{\substack{i,j=1\\i<j,i\sim j}}^n\biggl[ \biggl(\sum_{\substack{k=1\\k\sim i, k\nsim j}}^n |G_k|+1\biggr)\biggl(\sum_{\substack{k=1\\k\nsim i, k\sim j}}^n |G_k|+1\biggr)\biggr]\mid G_i\mid \mid G_j\mid.
		\end{flalign*}
		
		This completes the proof. 
	\end{proof}
	
	In the next section, as an application to \Cref{T1}, we calculate the Szeged index of the power graph of the finite cyclic group $\mathbb{Z}_n$ for $n>2$.
	We also determine the Szeged index of the power graph of the dihedral group $\mathrm{D}_n$ for  $n>2$.
	
	\section{Szeged Index of Power Graph of Finite Groups}
	
	The notion of  \textit{directed power graph} of a semigroup $S$ was introduced by Kelarav and Quinn in \cite{kelarev2002directed} as the digraph $\mathcal{P}(S)$ with vertex set $S$ in which there is an arc from one vertex $u$ to another vertex $v$ if  $v = u^m$ for some $m\in \mathbb{N}$.
	Motivated by this, Chakrabarty et al. in \cite{chakrabarty2009undirected} defined the undirected \textit{power graph} $\mathcal{P}(S)$ of a semi-group $S$ as an undirected graph whose vertex set is $S$ and two distinct vertices $u,v\in S$  are adjacent if  $u=v^m$ or $v=u^n$ for some $m,n\in \mathbb{N}$.
	In \cite{chakrabarty2009undirected} it was further shown that  $\mathcal{P}(G)$ is connected for any finite group $G$, and it is complete if and only if $G$ is a cyclic group of order $1$ or $p^m$ where $p$ is a prime and $m\in \mathbb{N}$.
	The power graph has received widespread attention from researchers over the recent years. Cameron and Ghosh have shown in \cite{cameron2011power} that if two finite abelian groups $G_1$, $G_2$ have isomorphic power graphs, then $G_1$ and $G_2$ are isomorphic to each other. 
	In \cite{cameron2010power2} Cameron has proved  that if two finite groups have isomorphic power graphs, then they
	have the same number of elements of each order.
	The spectral properties of power graphs of finite groups have been studied in \cite{banerjee2023reduced,banerjee2023spectra,banerjee2021spectra,banerjee2020signless}.
	For more interesting results on power graphs, the readers may refer to \cite{abawajy2013power}. 
	
	\subsection{Szeged Index of Power Graph of $\mathbb{Z}_n$}
	\label{S2}
	In this section, we derive a formula for the Szeged index of the power graph of $\mathbb{Z}_n$ for any $n>2$.
	Before getting into the details, we first throw some light on the structure of $\mathcal{P}(\mathbb{Z}_n)$.

	We shall denote the prime decomposition of $n$ by $n=p_1^{\alpha_1}p_2^{\alpha_2}\cdots p_m^{\alpha_m}$, where $p_1<p_2<\cdots <p_m$ are primes arranged in increasing order and $\alpha_i's$ are positive integers. 
	An element $d$ is said to be a \textit{proper divisor} of $n$ if $d\mid n$ and $d\notin\{1,n\}$.
	Given $n\in \mathbb{N}$, the total number of positive divisors of $n$ is given as follows: 
	\begin{flalign*}
	D'=(\alpha_1+1)(\alpha_2+1)\cdots(\alpha_m+1).
	\end{flalign*}
	The number of  \textit{positive proper divisors} of $n$ is given as follows: 
	\begin{flalign}\label{div}
	D=(\alpha_1+1)(\alpha_2+1)\cdots(\alpha_m+1)-2.
	\end{flalign}

	Suppose $d_1<d_2<\cdots< d_D$ be the set of all positive proper divisors of $n$ arranged in increasing order.

	Assume that  $\mathcal{G}$ is a graph  with vertex set  $V(\mathcal{G})=\{1,d_1,d_2,\dots,d_D\}$.
	Moreover, the adjacency criterion in $\mathcal{G}$ is given as follows:
	
	\paragraph{
		$1\sim d_i$ for all $1\le i\le D$, and $d_i\sim d_j$ if and only if $d_i\mid d_j$.
	}
	
	Using the above information, we have the following result about the structure of $\mathcal{P}(\mathbb Z_n)$.
	
	\begin{theorem}\cite[Theorem $2.2$]{mehranian2016note}
		\label{T2}
		If $n> 2$, then $\mathcal{P}(\mathbb Z_n)=\mathcal{G}[K_\ell,K_{\varphi(d_1)},K_{\varphi(d_2)},\dotso,K_{\varphi(d_D)}]$, where $d_1,d_2,\ldots,d_D$ form the set of all positive proper divisors of $n$, $D$ is given by \Cref{div} and $\ell=\varphi(n)+1$.
        Here, $\varphi(n)$ denotes the Euler's totient function, which counts the positive integers up to a given integer $n$ that are relatively prime to $n$.
	\end{theorem}
	
	We now use \Cref{T1,T2} to compute the Szeged index of $\mathcal{P}(\mathbb Z_n)$.

	\begin{theorem}
		\label{T3}
		If $n> 2$, then the Szeged index of $\mathcal{P}(\mathbb Z_n)$ is given as follows:
		
		\begin{flalign*}
		Sz(\mathcal{P}(\mathbb Z_n))&=\binom{\ell}{2}+\sum_{i=1}^{D}\binom{\varphi(d_i)}{2}+\sum_{i=1}^D\bigg(n-\ell- {\varphi(d_i)}+1\bigg) {\varphi(d_i)}\times \ell 
		\\
		&+\sum_{\substack{i,j=1\\i<j,d_i\sim d_j}}^D\biggl[ \biggl(\sum_{\substack{k=1\\d_k\sim d_i, d_k\nsim d_j}}^D {\varphi(d_k)}+1\biggr)\biggl(\sum_{\substack{k=1\\d_k\nsim d_i, d_k\sim d_j}}^D {\varphi(d_k)}+1\biggr)\biggr] {\varphi(d_i)}{\varphi(d_j)}
		\end{flalign*}
	\end{theorem}
	
	\begin{proof}
		We note that $Sz(K_\ell)=\binom{\ell}{2}$.
		Moreover, if $d_i>2$, then $Sz(K_{\varphi(d_i)})=\binom{\varphi(d_i)}{2}$ for $1\le i\le D$.
		
		We define $S=\{x\in \mathbb{Z}_n:\gcd(x,n)=1\}\cup \{0\}$.
		Consider an  edge $e=ab$ of $\mathcal{P}(\mathbb Z_n)$ such that  $a\in S$ and $b\notin S$.
		Since the vertices of $S$ are adjacent to every other vertex of $\mathcal{P}(\mathbb Z_n)$, we observe that $d(x,a)=1$ for all $x\neq a$. Thus, $d(x,a)>d(x,b)$ is true only when $x=b$.
		
		Hence,
		\begin{flalign*}
		n_2(ab|\mathcal{P}(\mathbb Z_n))=\lvert\{x:d(x,a)>d(x,b)\}\rvert=1.
		\end{flalign*}
		
		It follows that 
		\begin{align}
		\label{7}
		\begin{split}
		\sum_{\substack{e(=ab)\in E(\mathcal{P}(\mathbb Z_n))\\
				a\in S,b\notin S}} n(ab|\mathcal{P}(\mathbb Z_n))
		&=\sum_{\substack{e(=ab)\in E(\mathcal{P}(\mathbb Z_n))\\
				a\in S,b\notin S}} n_1(ab|\mathcal{P}(\mathbb Z_n))
		\\
		&=	\sum_{\substack{e(=ab)\in E(\mathcal{P}(\mathbb Z_n))\\
				a\in S,b\notin S}}\lvert\{x:d(x,a)<d(x,b)\}\rvert.
		\end{split}
		\end{align}
		
		Since  $b\notin S$, so $b\in K_{\varphi(d_i)}$ for some $1\le i\le D$.
		In order for $x$ to satisfy $d(x,a)<d(x,b)$, $x$ must be $a$ or $x$ must be in $K_{\phi(d_k)}$ such that $d_k\nsim d_i$.
		
		Thus using \Cref{7} we get,
		\begin{align}\label{8}
		\begin{split}
		\sum_{\substack{e(=ab)\in E(\mathcal{P}(\mathbb Z_n))\\
				a\in S,b\notin S}} n(ab|\mathcal{P}(\mathbb Z_n))
		&=\sum_{\substack{i=1}}^D\biggl[ \biggl(\sum_{\substack{k=1\\d_k\nsim d_i}}^D {\varphi(d_k)}+1\biggr)\biggr]\times \ell\times \varphi(d_i).
		\end{split}
		\end{align}
		
		Using \Cref{T1} and \Cref{8} we have
		\begin{flalign*}
			Sz(\mathcal{P}(\mathbb{Z}_n)) 
			&= \binom{\ell}{2} 
			+ \sum_{i=1}^{D} \binom{\varphi(d_i)}{2} 
			+ \sum_{i=1}^D \biggl[ \biggl(\sum_{\substack{k=1 \\ d_k \nsim d_i}}^D \varphi(d_k) + 1 \biggr) \biggr] \times \ell \times \varphi(d_i) \\
			&\quad + \sum_{\substack{i,j=1 \\ i < j, d_i \sim d_j}}^D 
			\biggl[ 
			\biggl(\sum_{\substack{k=1 \\ d_k \sim d_i \\ d_k \nsim d_j}}^D |K_{\varphi(d_k)}| + 1\biggr) 
			\biggl(\sum_{\substack{k=1 \\ d_k \nsim d_i \\ d_k \sim d_j}}^D |K_{\varphi(d_k)}| + 1\biggr) 
			\biggr] |K_{\varphi(d_i)}| |K_{\varphi(d_j)}| \\
			&= \binom{\ell}{2} 
			+ \sum_{i=1}^{D} \binom{\varphi(d_i)}{2} 
			+ \sum_{i=1}^D \biggl[ \biggl(\sum_{\substack{k=1 \\ d_k \nsim d_i}}^D \varphi(d_k) + 1 \biggr) \biggr] \times \ell \times \varphi(d_i) \\
			&\quad + \sum_{\substack{i,j=1 \\ i < j, d_i \sim d_j}}^n 
			\biggl[ 
			\biggl(\sum_{\substack{k=1 \\ d_k \sim d_i, d_k \nsim d_j}}^n \varphi(d_k) + 1 \biggr) 
			\biggl(\sum_{\substack{k=1 \\ d_k \nsim d_i, d_k \sim d_j}}^n \varphi(d_k) + 1 \biggr) 
			\biggr] \varphi(d_i) \varphi(d_j).
		\end{flalign*}
		
	\end{proof}
	We now use \Cref{T3} to determine the Szeged index of $\mathcal{P}(\mathbb Z_n)$ for some specific forms of $n$. 
	\begin{proposition}
		If $n=p^m$ where $p$ is a prime and $m\in \mathbb{N}$, then $Sz(\mathcal{P}(\mathbb Z_n))=\binom{n}{2}.$
	\end{proposition}
	
	\begin{proof}
		We know that if $n=p^m$, then $\mathcal{P}(\mathbb Z_n)$ is a complete graph (\cite{chakrabarty2009undirected}).
		Hence the result follows.
	\end{proof}
	\begin{proposition}
		If $n=pq$, where $p,q$ are prime numbers with $p<q$, then 
		\begin{flalign*}
		Sz(\mathcal{P}(\mathbb Z_n))=\frac{\ell(\ell-1)}{2}+\frac{(p-1)(p-2)}{2}+\frac{(q-1)(q-2)}{2}+\ell (2pq-(p+q)).
		\end{flalign*}
	\end{proposition}

	\begin{proof}
		If $n=pq$, then the positive proper divisors of $n$ are $p,q$.
		Thus, $d_1=p,d_2=q$, $D=2$ and $\ell=\varphi(pq)+1$. Since $p\nmid q$, we note that $d_1\nsim d_2$.
		Using \Cref{T2} we find that  $\mathcal{P}(\mathbb{Z}_n)$ is the $\mathcal{G}$-join of $K_{\ell},K_{\phi(p)},K_{\phi(q)}$ where $\mathcal{G}$ is shown in \Cref{F1}.
		\begin{figure}[H]
			
			\centering
			\begin{tikzpicture}
			\node[shape=circle,draw=black] (1) at (-2,0) {$p$};  
			\node[shape=circle,draw=black] (2) at (0,0)  {$1$};   
			\node[shape=circle,draw=black] (3) at (2,0)  {$q$};   
			
			\draw (1) -- (2);  
			\draw (2) -- (3);

			\end{tikzpicture}
			\caption{$\mathcal G=P_3$ for $\mathcal{P}(\mathbb{Z}_{pq})$}
			\label{F1}
		\end{figure}
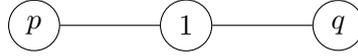
		
		Using \Cref{T3} we have
		
		\begin{flalign*}
		Sz(\mathcal{P}(\mathbb{Z}_{pq}))&=\binom{\ell}{2}+\binom{\varphi(p)}{2}+\binom{\varphi(q)}{2}
		+\bigg((\varphi(q)+1)\times \ell\times\varphi(p)\bigg)
		\\
		&+\bigg((\varphi(p)+1)\times \ell\times\varphi(q)\bigg)
		\\
		&=\frac{\ell(\ell-1)}{2}+\frac{(p-1)(p-2)}{2}+\frac{(q-1)(q-2)}{2}+q\ell(p-1)+p\ell(q-1)
		\\
		&=\frac{\ell(\ell-1)}{2}+\frac{(p-1)(p-2)}{2}+\frac{(q-1)(q-2)}{2}+\ell(2pq-(p+q)).
		\end{flalign*}

	\end{proof}
	
	\begin{proposition}
		If $n=pq^2$, then the Szeged index of $\mathcal{P}(\mathbb{Z}_n)$ is given by \begin{flalign*}
		Sz(\mathcal{P}(\mathbb{Z}_n))&=\binom{\ell}{2}+\binom{\varphi(p)}{2}+\binom{\varphi(q)}{2}+\binom{\varphi(q^2)}{2}+\binom{\varphi(pq)}{2}
		\\
		&+\ell\biggl(2pq^3 - 2pq^2 + 3pq - 2p - 2q^3 + 3q^2 - 3q + 1\biggr)
		\\
		&+p^2q^4-3p^2q^3+5p^2q^2-4p^2q + p^2 - 3pq^2 + 4pq - p - q^4 + 4q^3 - 4q^2 + q.
		\end{flalign*}
	\end{proposition}

	\begin{proof}
		If $n=pq^2$, then the positive proper divisors of $n$ are $p,q,q^2,pq$.
		Here $d_1=p,d_2=q,d_3=q^2,d_4=pq$ and $D=4$. 
		Also $\ell=\varphi(n)+1=q(p-1)(q-1)+1$.
		Using \Cref{T2} we find that  $\mathcal{P}(\mathbb{Z}_n)$ is the $\mathcal{G}$-join of $K_{\ell},K_{\phi(p)},K_{\phi(q)},K_{\phi(pq)},K_{\phi(q^2)}$ where $\mathcal{G}$ is shown in \Cref{F2}.
		\begin{figure}[H]
			\centering
			\begin{tikzpicture}
			\node[shape=circle,draw=black] (2) at (-2,0) {$p$};  
			\node[shape=circle,draw=black] (1) at (0,0)  {$1$};   
			\node[shape=circle,draw=black] (3) at  (2,0) {$q$};  
			\node[shape=circle,draw=black] (4) at  (0,2) {$pq$};  
			
			\node[shape=circle,draw=black] (5) at  (2,2) {$q^2$};

			\draw (1) -- (2);  \draw (1) -- (5);
			\draw (1) -- (3);  
			\draw (1) -- (4);  \draw (4) -- (3);  \draw (2) -- (4);\draw (3) -- (4);  
			\draw (3) -- (5);  
			
			\end{tikzpicture}
			\caption{$\mathcal{G}$ for $\mathcal{P}(\mathbb{Z}_{pq^2})$}
			\label{F2}	
		\end{figure}
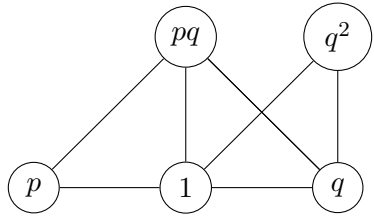
		
		Using \Cref{T3} we obtain
		\begin{align}\label{9}
		\begin{split}
		Sz(\mathcal{P}(\mathbb{Z}_{n}))&=\binom{\ell}{2}+\binom{\varphi(p)}{2}+\binom{\varphi(q)}{2}+\binom{\varphi(q^2)}{2}+\binom{\varphi(pq)}{2}\\
		&+\sum_{\substack{i=1}}^D\biggl[ \biggl(\sum_{\substack{k=1\\d_k\nsim d_i}}^D {\varphi(d_k)}+1\biggr)\biggr]\times \ell\times \varphi(d_i)+X
		\end{split}
		\end{align}
		where,
		
		\begin{flalign*}
		X=&\sum_{\substack{i,j=1\\i<j,d_i\sim d_j}}^D\biggl[ \biggl(\sum_{\substack{k=1\\d_k\sim d_i, d_k\nsim d_j}}^D {\varphi(d_k)}+1\biggr)\biggl(\sum_{\substack{k=1\\d_k\nsim d_i, d_k\sim d_j}}^D {\varphi(d_k)}+1\biggr)\biggr] {\varphi(d_i)}{\varphi(d_j)}.
		\end{flalign*}
		
		Now,
		\begin{align}\label{10}
		\begin{split}
		&\sum_{\substack{i=1}}^D\biggl[ \biggl(\sum_{\substack{k=1\\d_k\nsim d_i}}^D {\varphi(d_k)}+1\biggr)\biggr]\times \ell\times \varphi(d_i)
		\\
		&=
		\ell\biggl[\bigg(\varphi(q)+\varphi(q^2)+1\bigg)\times \varphi(p)
		+\bigg(\varphi(p)+1\bigg)\varphi(q)
		\\
		&+\biggl(\varphi
		(p)+\varphi(pq)+1\biggr)\times \varphi(q^2)+\biggl(\varphi(q^2)+1\biggr)\times \varphi(pq)\biggr]
		\\
		&=\ell\biggl[q^2(p-1)+p(q-1)+q(q-1)\bigg(p+(p-1)(q-1)\bigg)+(p-1)(q-1)\bigg(q(q-1)+1\bigg)\biggr]
		\\
		&=\ell\biggl[2pq^3 - 2pq^2 + 3pq - 2p - 2q^3 + 3q^2 - 3q + 1\biggr].
		\end{split}
		\end{align}
		Moreover, $p\sim pq, q\sim q^2, q\sim pq.$
		Thus we have,
		
		\begin{align}\label{11}
		\begin{split}
		X=&\sum_{\substack{i,j=1\\i<j,d_i\sim d_j}}^D\biggl[ \biggl(\sum_{\substack{k=1\\d_k\sim d_i, d_k\nsim d_j}}^D {\varphi(d_k)}+1\biggr)\biggl(\sum_{\substack{k=1\\d_k\nsim d_i, d_k\sim d_j}}^D {\varphi(d_k)}+1\biggr)\biggr] {\varphi(d_i)}{\varphi(d_j)}
		\\
		&=\biggl[\bigg((\varphi(q)+1)\varphi(p)\varphi(pq)\bigg)+\bigg((\varphi(pq)+1)\varphi(q)\varphi(q^2)\bigg)
		\\
		&+\bigg((\varphi(q^2)+1)(\varphi(p)+1)\varphi(q)\varphi(pq)\bigg)\biggr]
		\\
		&=q(p-1)^2(q-1)+q(q-1)^2\bigg((p-1)(q-1)+1\bigg)+p(p-1)(q-1)^2\bigg(q(q-1)+1\bigg)
		\\
		&=p^2q^4-3p^2q^3+5p^2q^2-4p^2q + p^2 - 3pq^2 + 4pq - p - q^4 + 4q^3 - 4q^2 + q.
		\end{split}
		\end{align}
		
		Substituting the values of \Cref{10,11} in \Cref{9} we have,
		
		\begin{flalign*}
		Sz(\mathcal{P}(\mathbb{Z}_n))&=\binom{\ell}{2}+\binom{\varphi(p)}{2}+\binom{\varphi(q)}{2}+\binom{\varphi(q^2)}{2}+\binom{\varphi(pq)}{2}
		\\
		&+\ell\biggl(2pq^3 - 2pq^2 + 3pq - 2p - 2q^3 + 3q^2 - 3q + 1\biggr)
		\\
		&+p^2q^4-3p^2q^3+5p^2q^2-4p^2q + p^2 - 3pq^2 + 4pq - p - q^4 + 4q^3 - 4q^2 + q.
		\end{flalign*}
		Thus the result follows.
	\end{proof}

	\subsection{Szeged Index of Power Graph of Dihedral Group}
	\label{S3}
	In this section, we calculate the Szeged index of the power graph of the dihedral group $\mathrm{D}_n$ for $n>2$.
	We establish a relation between the Szeged index of the power graph of dihedral group $\mathrm{D}_n$, and the Szeged index of the power graph of finite cyclic group $\mathbb{Z}_n$. 
	
	\begin{definition}
		\label{presentation}
		The dihedral group $\mathrm{D}_n$ of order $2n$ is given by the following presentation:
		\begin{flalign*}
		\mathrm{D}_n=\langle r,s: r^n=s^2=1, rs=sr^{-1}\rangle .
		\end{flalign*}
	\end{definition}
	
	If $n=6$, $\mathcal{P}(\mathrm {D}_n)$ takes the following form(\Cref{F3}).
	Clearly for all $n>2$, $\mathcal{P}(\mathrm {D}_n)$ contains $\mathcal{P}(\mathbb{Z}_n)$ as a subgraph.
	
	\begin{figure}[H]
		\centering
		\begin{tikzpicture}[scale=0.65]
		\node[shape=circle,draw=black] (0) at (0,4) {$0$};  
		\node[shape=circle,draw=black] (1) at (-2,2)  {$1$};   
		\node[shape=circle,draw=black] (5) at  (2,2) {$5$};  
		\node[shape=circle,draw=black] (2) at  (-2,-2) {$2$};  
		\node[shape=circle,draw=black] (4) at  (0,-4) {$4$};  
		\node[shape=circle,draw=black] (3) at  (2,-2) {$3$};

		\draw (0) -- (1); \draw (0) -- (5);	\draw (0) -- (2); \draw (0) -- (4);
		\draw (0) -- (3);

		\draw (1) -- (5);	\draw (1) -- (2); \draw (1) -- (4);
		\draw (1) -- (3); 
		
		\draw (2) -- (4); \draw (2) -- (5);	
		
		\draw (3) -- (5); \draw (4) -- (5); 
		
		\node[shape=circle,draw=black] (01) at (-2,4) {$s$};  
		\node[shape=circle,draw=black] (11) at (2,4)  {$sr^5$};   
		\node[shape=circle,draw=black] (51) at  (-1.5,7) {$sr$};  
		\node[shape=circle,draw=black] (21) at  (1.5,7) {$sr^4$};  
		\node[shape=circle,draw=black] (41) at  (-2.5,5.5) {$sr^2$};  
		\node[shape=circle,draw=black] (31) at  (2.5,5.5) {$sr^3$};  
		
		\draw (0) -- (01); \draw (0) -- (51);	\draw (0) -- (21); \draw (0) -- (41);
		\draw (0) -- (11); \draw (0) -- (31);

		\end{tikzpicture}
		\caption{$\mathcal{P}(\mathrm D_{6})$}
		\label{F3}	
	\end{figure}
	
	We now express the Szeged index of $\mathcal{P}(\mathrm{D}_n)$ in terms of the Szeged index of 
	$\mathcal{P}(\mathbb{Z}_n)$.
	
	\begin{theorem}
		\label{dg}
		If $n>2$, then the Szeged index of $\mathcal{P}(\mathrm{D}_n)$ is given as follows: 
		\begin{flalign*}
		Sz(\mathcal{P}(\mathrm{D}_n))
		&=Sz(\mathcal{P}(\mathbb{Z}_n^*))+\sum_{i=1}^D\bigg( n+1+\sum_{\substack{k=1\\d_k\nsim d_i}}^D \varphi(d_k)\bigg) \varphi(d_i)+n(2n-1+\varphi(n))+\varphi(n).
		\end{flalign*}
		Here, 
		$\mathbb{Z}_n^*=\mathbb{Z}_n\setminus \{0\}$,
		and $d_k\sim d_i$ only when $d_k\mid d_i$.
	\end{theorem}
	
	\begin{proof}

		We shall denote $\mathbb{Z}_n^*=\mathbb{Z}_n\setminus \{0\}$.
		We can obtain $\mathcal{P}(\mathbb{Z}_n^*)$ from $\mathcal{P}(\mathbb{Z}_n)$ by deleting the vertex $0$ and all edges incident on it from $\mathcal{P}(\mathbb{Z}_n)$.
		Now,
		\begin{align}\label{D}
		\begin{split}
		Sz(\mathcal{P}(\mathrm{D}_n))&=\sum_{e(=ab)\in E(\mathcal{P}(\mathrm{D}_n))} n(ab|\mathcal{P}(\mathrm{D}_n))
		\\
		&=\sum_{e(=ab)\in E(\mathcal{P}(\mathbb{Z}_n^*))} n(ab|\mathcal{P}(\mathrm{D}_n))
		+\sum_{e\in \tau} n(ab|\mathcal{P}(\mathrm{D}_n))+\sum_{e\in \sigma} n(ab|\mathcal{P}(\mathrm{D}_n))
		\end{split}
		\end{align}
		where $\tau=\{(0,t): t\in V(\mathcal{P}(\mathbb{Z}_n))\}\subseteq E(\mathcal{P}(\mathbb{Z}_n))$ and 	$\sigma=\{(0,sr^i): 0\le i\le n-1\}\subseteq E(\mathcal{P}(\mathrm{D}_n))$.
		
		Now, we shall evaluate each term of the summation given in \Cref{D} one by one in what follows.
		\begin{flalign*}
		\sum_{e(=ab)\in E(\mathcal{P}(\mathbb{Z}_n^*))} n(ab|\mathcal{P}(\mathrm{D}_n))&=
		\sum_{e(=ab)\in E(\mathcal{P}(\mathbb{Z}_n^*))} n_1(ab|\mathcal{P}(\mathrm{D}_n))n_2(ab|\mathcal{P}(\mathrm{D}_n))
		\end{flalign*}
		
		We observe that the vertex $0$ is adjacent to every other vertex of 
		$\mathcal{P}(\mathrm{D}_n)$. Moreover, every vertex of $\mathcal{P}(\mathbb{Z}_n^*)$ has a distance $2$ from all the vertices in the set $\{sr^i: 0\le i\le n-1\}$.
		
		Thus, we have,
		\begin{align}\label{12}
		\begin{split}
		n_1(ab|\mathcal{P}(\mathrm{D}_n))&=\lvert\{x\in V(\mathcal{P}(\mathrm{D}_n)) :d(x,a)<d(x,b)\}\rvert
		\\
		&=\lvert\{x\in V(\mathcal{P}(\mathbb{Z}_n^*)) :d(x,a)<d(x,b)\}\rvert
		\\
		&=n_1(ab|\mathcal{P}(\mathbb{Z}_n^*))
		\end{split}
		\end{align}
		
		Similarly, \begin{flalign}
		\label{13}
		n_2(ab|\mathcal{P}(\mathrm{D}_n))&=n_2(ab|\mathcal{P}(\mathbb{Z}_n^*))
		\end{flalign}
		
		Using \Cref{12,13} we have,
		\begin{flalign}\label{14}
		n(ab|\mathcal{P}(\mathrm{D}_n))&=n(ab|\mathcal{P}(\mathbb{Z}_n^*))
		\end{flalign}
		
		Again,	
		\begin{flalign*}
		\sum_{e\in \tau} n(ab|\mathcal{P}(\mathrm{D}_n))
		&=\sum_{e\in \tau}n_1(ab|\mathcal{P}(\mathrm{D}_n))\sum_{e\in \tau}n_2(ab|\mathcal{P}(\mathrm{D}_n))	
		\end{flalign*}
		
		Assume that $e=(a,b)\in \tau$. Then
		\begin{flalign*}
		n_1(ab|\mathcal{P}(\mathrm{D}_n))&=\lvert\{x\in \mathcal{P}(\mathrm{D}_n):d(x,a)<d(x,b)\}\rvert
		\\
		&=\lvert\{x\in \mathcal{P}(\mathrm{D}_n):d(x,0)<d(x,t)\}\rvert \text{ where } t\in V(\mathcal{P}(\mathbb{Z}_n)).
		\end{flalign*}

		We consider the following two mutually exclusive and exhaustive cases:
		\begin{itemize}
			\item [Case 1:] $t$ is a generator of $\mathbb{Z}_n$, i.e. $\gcd(t,n)=1$.
			
			In this case $x$ can be $0$, or $x$ can be one of the elements of the set $\{sr^i:0\le i\le n-1\}$.
			Thus, we get $n+1$ choices for $x$.
			\item [Case 2:] $t$ is not  a generator of $\mathbb{Z}_n$, i.e. $\gcd(t,n)\neq 1$.
			
			Suppose $x\in K_{\varphi(d_i)}$ for some positive proper divisor $d_i$ of $n$ where $1\le i\le D$.
			Then, clearly  $x$ can be $0$, or $x$ can be one of the elements of the set $\{sr^i:0\le i\le n-1\}$, or $x$ must be in those $K_{\varphi(d_k)}$ such that $d_i\nsim d_k$ where $1\le k\le D$.
			Thus, we have 
			$n+1+\sum_{\substack{k=1\\d_k\nsim d_i}}^D \varphi(d_k)$ $\text{ choices for } x.$

		\end{itemize}
		We further note that $n_2(ab|\mathcal{P}(\mathrm{D}_n))=1$ for all $(a,b)\in \tau$.
		Thus we obtain, 
		
		\begin{flalign}\label{15}
		\sum_{e\in \tau} n(ab|\mathcal{P}(\mathrm{D}_n))&=\sum_{i=1}^D\bigg( n+1+\sum_{\substack{k=1\\d_k\nsim d_i}}^D \varphi(d_k)\bigg) \varphi(d_i).
		\end{flalign}

		Again,
		\begin{flalign*}
		\sum_{e\in \sigma} n(ab|\mathcal{P}(\mathrm{D}_n))&=\sum_{e\in \sigma} n_1(ab|\mathcal{P}(\mathrm{D}_n)) n_2(ab|\mathcal{P}(\mathrm{D}_n))
		\end{flalign*}
		
		We note that
		\begin{align*}
		\begin{split}
		n_1(ab|\mathcal{P}(\mathrm{D}_n))&=\lvert\{x:d(x,a)<d(x,b)\}\rvert
		\\
		&=\lvert\{x:d(x,0)<d(x,sr^i)\}\rvert
		\\
		&=\lvert(\mathrm{D}_n\setminus\{sr^i\})\rvert
		\\
		&=2n-1.
		\end{split}
		\end{align*}
		
		We further note that $n_2(ab|\mathcal{P}(\mathrm{D}_n))=1$ for all $(a,b)\in \sigma$.
		
		Thus we have
		\begin{flalign}\label{21}
		\sum_{e\in \sigma} n(ab|\mathcal{P}(\mathrm{D}_n))=\sum_{i=1}^n 2n-1=n(2n-1).
		\end{flalign}
		
		Substituting the values obtained from  \Cref{14,15,21} in \Cref{D} we have,
		
		\begin{flalign*}
		Sz(\mathcal{P}(\mathrm{D}_n))&=\sum_{e(=ab)\in E(\mathcal{P}(\mathbb{Z}_n^*))} n(ab|\mathcal{P}(\mathrm{D}_n))
		+\sum_{e\in \tau} n(ab|\mathcal{P}(\mathrm{D}_n))+\sum_{e\in \sigma} n(ab|\mathcal{P}(\mathrm{D}_n))
		\\
		&=Sz(\mathcal{P}(\mathbb{Z}_n^*))+(n+1)\varphi(n)+\sum_{i=1}^D\bigg( n+1+\sum_{\substack{k=1\\d_k\nsim d_i}}^D \varphi(d_k)\bigg) \varphi(d_i)+n(2n-1)
		\\
		&=Sz(\mathcal{P}(\mathbb{Z}_n^*))+\sum_{i=1}^D\bigg( n+1+\sum_{\substack{k=1\\d_k\nsim d_i}}^D \varphi(d_k)\bigg) \varphi(d_i)+n(2n-1+\varphi(n))+\varphi(n).
		\end{flalign*}
		
		Thus the result follows.
	\end{proof}
	
	We end this section with the following result.
	\begin{proposition}
		If $n=pq$ where $p,q$ are distinct primes with $p<q$, then the Szeged index of $Sz(\mathcal{P}(\mathrm{D}_n))$ is given as follows:
		\begin{flalign*}
		Sz(\mathcal{P}(\mathrm{D}_n))&=Sz(\mathcal{P}(\mathbb{Z}_n^*))+p(q^2-1)+q(p^2-1)+pq(3pq-p-q)+(p-1)(q-1).
		\end{flalign*}
		
	\end{proposition}
	
	\begin{proof}
		Since $n=pq$, the only positive proper divisors of $n$ are $p$ and $q$. Thus $D=2$, $d_1=p$, and $d_2=q$.
		Moreover, since $p\nmid q$, so $d_1\nsim d_2$.
		Hence, using \Cref{dg}, we have
		
		\begin{flalign*}
		Sz(\mathcal{P}(\mathrm{D}_n)){}=&Sz(\mathcal{P}(\mathbb{Z}_n^*))+\sum_{i=1}^2\bigg( n+1+\sum_{\substack{k=1\\d_k\nsim d_i}}^2 \varphi(d_k)\bigg) \varphi(d_i)+n(2n-1+\varphi(n))+\varphi(n)
		\\
		=&Sz(\mathcal{P}(\mathbb{Z}_n^*))+( n+1+ \varphi(d_1)) +( n+1+ \varphi(d_2))\varphi(d_1)
		\\
		&+n(2n-1+\varphi(n))+\varphi(n)
		\\
		=&Sz(\mathcal{P}(\mathbb{Z}_n^*))+(pq+1+(p-1))(q-1)+(pq+1+(q-1))(p-1)
		\\
		&+pq(3pq-p-q)+(p-1)(q-1)
		\\
		=&Sz(\mathcal{P}(\mathbb{Z}_n^*))+p(q^2-1)+q(p^2-1)+pq(3pq-p-q)+(p-1)(q-1).
		\end{flalign*}
		
	\end{proof}

	\section{SAGE CODE}
	\label{S4}
	In this section, we provide the corresponding SAGE codes for calculating the Szeged index of $\mathcal{P}(\mathbb{Z}_n)$ and $\mathcal{P}(\mathrm{D}_n)$.
	On providing a given value of $n$ in the codes provided below,  Code \ref{1F} gives a plot of $\mathcal{P}(\mathbb{Z}_n)$, and  Code \ref{2F} gives the Szeged index of $\mathcal{P}(\mathbb{Z}_n)$.
	Moreover, Code \ref{3F} gives a plot of $\mathcal{P}(\mathrm{D}_n)$, and  Code \ref{4F} gives the Szeged index of $\mathcal{P}(\mathrm{D}_n)$.
	
	\vspace{20mm}
	\begin{lstlisting}[caption={Plot of $\mathcal{P}(\mathbb{Z}_n).$},label={1F},language=Python]
	G=Graph()
	E=[]
	n=21
	for i in range (n):
	for j in range (n):
	if(i!=j):
	for k in range(n):
	if(((k*i)%n==j) or ((k*j)%n==i)):
	E.append((i,j))
	G.add_edges(E)
	G.plot()
	\end{lstlisting}

	\begin{lstlisting}[caption={Szeged index of  $\mathcal{P}(\mathbb{Z}_n)$.},label={2F},language=Python]
	
	G=Graph()
	E=[]
	n=21
	for i in range (n):
	for j in range (n):
	if(i!=j):
	for k in range(n):
	if(((k*i)%n==j) or ((k*j)%n==i)):
	E.append((i,j))
	G.add_edges(E)
	S=G.szeged_index()
	print("Szeged_index_of_power_graph_of_Zn=",S)
	\end{lstlisting}
	
	\begin{lstlisting}[caption={Plot of $\mathcal{P}(\mathrm{D}_n).$},label={3F},language=Python]
	
	G=Graph()
	E=[]
	n=10
	for i in range (n):
	for j in range (n):
	if(i!=j):
	for k in range(n):
	if(((k*i)%n==j) or ((k*j)%n==i)):
	E.append((i,j))
	G.add_edges(E)
	for i in range(n,2*n):
	G.add_vertex(i)
	G.add_edge([0,i])
	G.plot()
	\end{lstlisting}
	
	\begin{lstlisting}[caption={Szeged index of $\mathcal{P}(\mathrm{D}_n).$},label={4F},language=Python]
	
	G=Graph()
	E=[]
	n=10
	for i in range (n):
	for j in range (n):
	if(i!=j):
	for k in range(n):
	if(((k*i)%n==j) or ((k*j)%n==i)):
	E.append((i,j))
	G.add_edges(E)
	for i in range(n,2*n):
	G.add_vertex(i)
	G.add_edge([0,i])
	S=G.szeged_index()
	print("Szeged index of power graph of dihedral group=",S)
	\end{lstlisting}

	\section{Conclusion}
	In this paper, we have studied the Szeged index of the power graph of finite cyclic and dihedral groups. We first obtained a formula for the Szeged index of the generalized join of graphs in terms of the constituent graphs. 
	Using it, we first determined the Szeged index of the power graph of finite cyclic groups.
	We further obtain a relation between the Szeged index of the power graph of dihedral groups and the Szeged index of the power graph of finite cyclic groups.
	Finally, we provide  SAGE codes for evaluating the  Szeged indices of the power graph of finite cyclic groups and dihedral groups.
	
	\section{Conflict of Interest}
	The author states that there is no conflict of interest.
	
\bigskip

\end{document}